\documentclass[11pt]{article}

\usepackage{fullpage}
\usepackage{amsmath, amsthm, amsfonts, amssymb, amstext, mathrsfs, enumerate}
\usepackage{graphicx, ragged2e, lscape, framed, xcolor}
\usepackage{subfiles}

\theoremstyle{plain}
\newtheorem{theorem}{Theorem}[section]

\newtheorem{conjecture}[theorem]{Conjecture}

\newtheorem{claim}{Claim}[section]

\numberwithin{equation}{section}
\allowdisplaybreaks

\newcommand{\affl}[3]{\noindent #1, Email: {\tt #2}\\ \textsc{#3}\\[1.5pt]}

\usepackage[pagebackref]{hyperref}
\hypersetup{
	colorlinks=true,
    urlcolor=purple,
	linkcolor=purple,
    citecolor=purple,
}

\DeclareMathOperator{\dist}{dist}
\DeclareMathOperator{\diag}{diag}
\DeclareMathOperator{\supp}{supp}
\def\p{\mbox{\boldmath $p$}}
\def\q{\mbox{\boldmath $q$}}
\def\x{\mbox{\boldmath $x$}}
\def\y{\mbox{\boldmath $y$}}

\def\0{\mbox{\boldmath $0$}}
\def\1{\mbox{\boldmath $1$}}

\title{\textbf{Spectral gap of the normalized distance Laplacian}}
\author{Hitesh Kumar \and Kamal Lochan Patra}
\date{}

\begin{document}
\maketitle
\begin{abstract}
The smallest positive eigenvalue $\partial_2$ of the normalized distance Laplacian matrix $\mathcal{D}^{\mathcal{L}}$ of a connected graph is called its \emph{spectral gap} and is intimately related to the Cheeger constant of $\mathcal{D}^{\mathcal{L}}$. Byrne, Johnston, Schildkraut and Tait (2025) conjectured that 
 \[ \partial_2 \ge \frac{2}{3}\]
for all connected graphs. We prove the following stronger result: for any connected graph $G$ of order at least 2,
\[\partial_2 \ge \frac{2}{3} + \frac{4}{3\,t_{\max}},\]
where $t_{\max}$ denotes the maximum transmission in $G$. Moreover, equality holds if and only if $G\cong K_{m,m}$ for some $m\ge 1$.  
\end{abstract}



\section{Introduction}

\subsection{Notation and terminology}

Let $G = (V(G), E(G))$ be a finite simple \emph{connected} graph of order $n = |V(G)|\ge 2$. For a pair of vertices $u, v\in V(G)$, let $\dist(u,v)$ denote the distance between $u$ and $v$ in $G$. Note that $\dist(u,u)=0$ and $\dist(u,v)=1$ if $u\sim v$ in $G$. The \emph{distance matrix} of $G$, denoted $\mathcal{D}(G)$, is the $n\times n$ matrix whose entries are given by 
\[ \mathcal{D}(G)_{u,v}:= \dist(u,v).\]
The \emph{transmission} of a vertex $u\in V(G)$ is defined to be 
\[ t(u):=\sum_{v\in V(G)}\dist(u,v).\]
In words, $t(u)$ is the row-sum of $D(G)$ corresponding to the vertex $u$. Since $G$ is connected, $t(u)>0$ for all $u\in V(G)$. The \emph{transmission matrix} of $G$ is the diagonal matrix
\[ T(G):=\diag(t(u))_{u\in V(G)}.\]
We denote the \emph{maximum transmission} in $G$ by
\[t_{\max}(G): = \max_{u\in V(G)} t(u).\]
The \emph{normalized distance Laplacian matrix} of $G$ is the matrix
\[ \mathcal{D}^{\mathcal{L}}(G):= T(G)^{-1/2}\left(T(G)- \mathcal{D}(G)\right)T(G)^{-1/2}.\]
Since this is a real symmetric matrix, the eigenvalues are real, which we enumerate as
\[ \partial_1(G)\le \partial_2(G)\le \ldots\le \partial_n(G).\]
It is easy to see that the quadratic form
\begin{equation}\label{eq:quadratic_form}
  \x^T (T(G)-\mathcal{D}(G))  \x  = \frac{1}{2} \sum_{u,v\in V(G)}\dist(u,v)(x_u-x_v)^2\ge 0  
\end{equation}
for any $\x\in \mathbb{R}^{V(G)}$, where the summation is over ordered pair of vertices. Moreover, equality holds in \eqref{eq:quadratic_form} if and only if $\x$ is a constant vector. Thus, $\partial_1(G) = 0$ and $\partial_2(G) > 0$. The eigenvalue $\partial_2(G)$ is called the \emph{spectral gap} of $\mathcal{D}^{\mathcal{L}}(G)$. 

\subsection{Motivation and main results}

Distance matrices of connected graphs are well-studied in the literature. We refer the reader to the seminal work on distance matrices of Graham--Pollak \cite{Graham_Pollak_1971} and Graham--Lov\'{a}sz \cite{Graham_Lovasz_1978}. Also refer to the survey on distance spectra by Aouchiche--Hansen \cite{Aouchiche_Hansen_2014}.

The classical Laplacian and the signless Laplacian matrices (see \cite{Chung_1997}) were generalized to the distance setting by Aouchiche--Hansen \cite{Aouchiche_Hansen_2013} using the transmission matrix. We refer to the recent surveys on distance Laplacian \cite{Rather_Aouchiche_2025} and distance signless Laplacian \cite{Rather_Ganie_Wang_2026}.  

Motivated by the extensive work on normalized Laplacian matrices (see \cite{Chung_1997}), Reinhart \cite{Reinhart_2021} (cf. \cite{Hogben_Reinhart_2022}) investigated the spectra of the normalized distance Laplacian matrix $\mathcal{D}^{\mathcal{L}}$ and contrasted its behaviour with the spectra of the normalized Laplacian. See \cite{Ganie_Rather_Das_2023, Johnston_Tait_2024} for recent work on $\mathcal{D}^{\mathcal{L}}$. 

Byrne--Johnston--Schildkraut--Tait \cite{Byrne_Johnston_Schildkraut_Tait_2025} investigated \emph{expansion} properties of $\mathcal{D}^{\mathcal{L}}$. In particular, they obtained sharp lower bounds for the Cheeger constant $h_G$ of $\mathcal{D}^{\mathcal{L}}(G)$ and contrasted its behaviour with the classical Cheeger constant. A consequence of their result is that 
\begin{equation}
   h_G \ge \frac{1}{3}, 
\end{equation}
which is asymptotically sharp. By Cheeger's inequality (cf. \cite{Chung_1997}), it is known that 
\begin{equation}\label{eq:cheeger}
    \frac{h_G^2}{2}\le \partial_2(G)\le 2h_G.
\end{equation}
Thus, we immediately see that 
\[ \partial_2(G)\ge \frac{1}{18}\]
for any connected graph $G$. Byrne et al. \cite{Byrne_Johnston_Schildkraut_Tait_2025} improved this lower bound further. 

\begin{theorem}[Byrne--Johnston--Schildkraut--Tait \cite{Byrne_Johnston_Schildkraut_Tait_2025}] For any connected graph $G$, 
\[ \partial_2(G)\ge \frac{9-4\sqrt{2}}{7}\approx 0.478.\] 
\end{theorem}

Furthermore, they proposed the following conjecture. 

\begin{conjecture}[Byrne--Johnston--Schildkraut--Tait \cite{Byrne_Johnston_Schildkraut_Tait_2025}] 
\label{conj:spectral_gap} For any connected graph $G$, 
\[ \partial_2(G)\ge \frac{2}{3}.\]
\end{conjecture}

As some evidence, Byrne et al. \cite{Byrne_Johnston_Schildkraut_Tait_2025} proved Conjecture \ref{conj:spectral_gap} for Cayley graphs on abelian groups.

Our main result is the following sharp lower bound for $\partial_2$, which resolves Conjecture \ref{conj:spectral_gap}.

\begin{theorem}\label{thm:spectral_gap_main}
For any connected graph $G$ of order $n\ge 2$, we have 
    \[\partial_2(G)\ge \frac{2}{3} + \frac{4}{3t_{\max}(G)}.\]
Moreover, equality holds if and only if $G\cong K_{m,m}$ for some $m\ge 1$.     
\end{theorem}

It is clear from the above theorem that the lower bound $\frac{2}{3}$ is never achieved by any finite graph. But $\frac{2}{3}$ is an asymptotically tight lower bound. Indeed, 
\[ \partial_2(K_{m,m}) = \frac{2}{3} + \frac{4}{3(3m-2)} \rightarrow \frac{2}{3} \]
as $m\rightarrow \infty$. We prove Theorem \ref{thm:spectral_gap_main} in Section \ref{sec:Proof}.

\subsection{Further discussion}

Our result can be easily extended to the more general setup of finite metric spaces. Let $(X,d)$ be a finite metric space with metric $d$, and define
\[ \delta(X,d) := \min_{\substack{u,v \, \in X\\u\neq v}} d(u,v)\quad \text{and}\quad t_{\max}(X,d) := \max_{u\in X} \, \sum_{v\in X}d(u,v).\]
Then, 
\[ \partial_2(X,d)\ge \frac{2}{3} + \frac{4\delta(X,d)}{3t_{\max}(X,d)},\]
where $\partial_2(X,d)$ denotes the least positive eigenvalue of the normalized distance Laplacian $\mathcal{D}^{\mathcal{L}}(X,d)$ defined in the natural way (cf. \cite{Byrne_Johnston_Schildkraut_Tait_2025}). 

Johnston--Tait \cite{Johnston_Tait_2024} asked how small $\partial_2$ can be over all graphs of order $n$. This is analogous to a problem of Aldous--Fill \cite{Aldous_Fill_2002} about the spectral gap of the normalized Laplacian, which was solved by Aksoy--Chung--Tait--Tobin \cite{Aksoy_Chung_Tait_Tobin_2018}. We refine the question of Johnston--Tait and propose the following conjecture for interested readers. 

\begin{conjecture} Let $G$ be a $\partial_2$-minimizer among all connected graphs of given order $n$. 
\begin{enumerate}[$(i)$]
    \item If $n = 2m$, then $G\cong K_{m,m}$.
    \item If $n=2m+1$, then $G\cong K_{m,m+1}$. 
\end{enumerate}
\end{conjecture}

As for the $\partial_2$-maximization, Reinhart \cite{Reinhart_2021} proved that for a connected graph of order $n$, 
\begin{equation}
    \partial_2(G)\le \frac{n}{n-1}\quad \text{and} \quad \partial_n(G)\ge \frac{n}{n-1}.
\end{equation}

Ganie--Rather--Das \cite{Ganie_Rather_Das_2023} and Johnston--Tait \cite{Johnston_Tait_2024} established that 
\begin{equation}
   \partial_2(G) = \frac{n}{n-1} \iff \partial_n(G) = \frac{n}{n-1} \iff G\cong K_n. 
\end{equation}

\section{Proof of Theorem \ref{thm:spectral_gap_main}}
\label{sec:Proof}

Let $G$ be a connected graph of order $n\ge 2$. Throughout this proof, we will omit $G$ from $V(G),\, \partial_2(G),\, T(G),\, \mathcal{D}(G), \, D^{\mathcal{L}}(G)$ and $t_{\max}(G)$. We denote the all-ones vector by $\1$. For a vector $\y = (y_u)\in \mathbb{R}^{V}$, let $|\y|$ denote the vector whose entries are absolute values of the entries of $\y$. The \emph{support} of $\y$ is the set $\supp(\y):=\{u\in V: y_u\neq 0\}$. The \emph{positive} and \emph{negative support} of $\y$ are defined to be $\supp^+(\y):=\{u\in V: y_u> 0\}$ and $\supp^-(\y):=\{u\in V: y_u< 0\}$, respectively. 

We first prove an inequality for a quadratic form involving $T$ and $D$. 
\begin{claim}\label{claim:T_D_inequality} Let $\y\in \mathbb{R}^{V}$ be such that $\y\neq \0$ and $\1^\top \y = 0$. Then 
\begin{equation}\label{eq:T_D_inequality}
    \y^\top T\y - 3\y^\top \mathcal{D}\y \ge 4 \|\y\|_2^2.
\end{equation}
\end{claim}

\begin{proof}
Write $\y = \p - \q$, where $\p= (p_u)_{u\in V}$ and $\q= (q_u)_{u\in V}$ such that for any $u\in V$, 
 \[p_u = \max\{y_u,0\}, \qquad q_u=\max\{-y_u, 0\}.\] 
Clearly, $\p, \q$ are non-negative vectors with disjoint support. Indeed, 
\[ \supp^+(\y) = \supp(\p) \quad \text{ and }\quad \supp^-(\y) = \supp(\q).\]
Since $\y$ is a non-zero vector orthogonal to $\1$, we have
\begin{equation}\label{eq:s}
   s:=\1^\top \p = \1^\top \q > 0. 
\end{equation}
Now, observe that 
\begin{align*}
   \y^\top \mathcal{D}\y & = \p^\top \mathcal{D}\p + \q^\top \mathcal{D}\q - 2\p^\top \mathcal{D}\q;\\
   |\y|^\top \mathcal{D}|\y| & = \p^\top \mathcal{D}\p + \q^\top \mathcal{D}\q + 2\p^\top \mathcal{D}\q = \y^\top \mathcal{D} \y + 4\p^\top \mathcal{D}\q;\\
   \y^\top T\y & = |\y|^\top T|\y|.
\end{align*}
Therefore, 
\begin{equation}\label{eq:expansion}
     \y^\top T\y - 3\y^\top \mathcal{D}\y =  |\y|^\top (T-\mathcal{D})|\y| + 2(2\p^\top \mathcal{D}\q - \p^T\mathcal{D}\p)+ 2(2\p^\top \mathcal{D}\q - \q^T\mathcal{D}\q).
\end{equation}
Now, using \eqref{eq:quadratic_form}, we have 
\begin{equation}\label{eq:absolute_y}
    |\y|^\top (T-\mathcal{D})|\y| \ge 0.
\end{equation}

Using \eqref{eq:s}, we see that 
\begin{align*}\label{eq:triangle_inequality}
        & \qquad s(2\p^\top \mathcal{D}\q - \p^T\mathcal{D}\p) \\
        & = (\1^\top \p)(2\p^\top \mathcal{D}\q) - (\1^\top \q) \p^T\mathcal{D}\p) \nonumber\\
        & = \sum_{\substack{u,w\in \supp(\p)\\v\in \supp(\q)}} p_u p_wq_v \left(\dist(u,v) + \dist(v,w) - \dist(u,w)\right)\\
        & = 2\sum_{\substack{u\in \supp(\p)\\v\in \supp(\q)}} p_u^2q_v \dist(u,v) + \sum_{\substack{u\neq w\in \supp(\p)\\v\in \supp(\q)}} p_u p_wq_v \left(\dist(u,v) + \dist(v,w) - \dist(u,w)\right) \\
        & \ge 2\sum_{\substack{u\in \supp(\p)\\v\in \supp(\q)}} p_u^2q_v\dist(u,v) \quad (\text{by triangle inequality})\\
        & \ge 2\sum_{\substack{u\in \supp(\p)\\v\in \supp(\q)}} p_u^2q_v \quad (\text{since }u\neq v, \dist(u,v)\ge 1)\\
        & = 2 \sum_{u\in \supp(\p)} p_u^2 \left(\sum_{v\in \supp(\q)} q_v\right) \\
        & = 2s\|\p\|^2_2.
\end{align*}
Thus, 
\begin{equation}\label{eq:p_norm}
    2\p^\top \mathcal{D}\q - \p^T\mathcal{D}\p \ge 2\|\p\|^2_2.
\end{equation}
Similarly, one can argue that 
\begin{equation}\label{eq:q_norm}
    2\p^\top \mathcal{D}\q - \q^T\mathcal{D}\q \ge 2\|\q\|^2_2.
\end{equation}
Combining \eqref{eq:expansion}, \eqref{eq:absolute_y}, \eqref{eq:p_norm} and \eqref{eq:q_norm}, we get
\[ 
\y^\top T\y - 3\y^\top \mathcal{D}\y \ge 4\|\p\|^2_2 + 4\|\q\|_2^2 = 4\|\y\|^2_2. 
\]
This completes the proof of the inequality.
\end{proof}

\begin{claim}\label{claim:inequality_proof} We have
    \[\partial_2\ge \frac{2}{3} + \frac{4}{3\,t_{\max}}.\]
\end{claim}

\begin{proof}
By the well-known Min-Max Theorem, we have 
\[ \partial_2 = \max_{\substack{U\\ \dim(U)=n-1}} \min_{ \0\ne \x\in U}\frac{\x^\top \mathcal{D}^{\mathcal{L}}\,\x}{\x^\top \x}.\]
Choose the $(n-1)$-dimensional subspace
\[ U  = T^{1/2}\1^{\perp} = \{T^{1/2}\y: y\in \mathbb{R}^{V}, \ \1^{\top}\y = 0\}\]
For $\x = T^{1/2}\y$, we have 
\[ \frac{\x^\top \mathcal{D}^{\mathcal{L}}\,\x}{\x^\top \x} = \frac{\y^\top (T-\mathcal{D})\y}{\y^\top T \y}.\]
Moreover, 
\begin{equation}\label{eq:norm_transmission}
    \y^\top T\y = \sum_{u\in V}t(u)y_u^2 \le t_{\max}\|\y\|_2^2,
\end{equation}
and for any $\0\neq  \y$ such that $\1^\top \y = 0$, we have 
\begin{align*}
   \frac{\y^\top (T-\mathcal{D})\y}{\y^\top T \y}
   & = \frac{2}{3} + \frac{\y^\top T\y - 3\y^\top \mathcal{D}\y}{3\y^\top T\y}\\
   & \ge \frac{2}{3} + \frac{4\|y\|_2^2}{3\y^\top T\y} \qquad (\text{by Claim \ref{claim:T_D_inequality}})\\
   & \ge \frac{2}{3} + \frac{4}{3\,t_{\max}}.
\end{align*}  
Thus,
\[    \partial_2  \ge  \min_{ \substack{\0\ne \y\in \mathbb{R}^{V}\\ \1^\top \y = 0}}\frac{\y^\top (T-\mathcal{D})\y}{\y^\top T \y}\ge \frac{2}{3} + \frac{4}{3\,t_{\max}}.\]
The claim follows.
\end{proof}

\begin{claim}\label{claim:equality_analysis} We have 
    \[\partial_2 = \frac{2}{3} + \frac{4}{3\,t_{\max}}\]
if and only if $G\cong K_{m,m}$ for some $m\ge 1$.     
\end{claim}

\begin{proof} (Sufficiency). First, assume that $G \cong K_{m,m}$. It is clear that 
\[t(u) = t_{\max}(K_{m,m})  = 3m-2\]
for all $u\in V(K_{m,m})$. Thus, the transmission matrix $T(K_{m,m}) = (3m-2)I$. 

Moreover, the spectrum of the distance matrix $\mathcal{D}(K_{m,m})$ is given by (cf. \cite{Aouchiche_Hansen_2014})
\[ (3m-2)^{(1)},\, (m-2)^{(1)},\, (-2)^{(2m-2)}.\]
As 
\[ \mathcal{D}^{\mathcal{L}}(K_{m,m}) = I - \frac{1}{3m-2}\mathcal{D}(K_{m,m}),\]
we see that  
\[ \partial_2(K_{m,m}) = 1 - \frac{m-2}{3m-2} = \frac{2}{3} + \frac{4}{3\,t_{\max}(K_{m,m})}.\]

(Necessity). Now, suppose equality holds for some $G$. From the proof of Claim \ref{claim:inequality_proof}, equality must hold in \eqref{eq:T_D_inequality}. Then, from the proof of Claim \ref{claim:T_D_inequality}, equality must hold in \eqref{eq:absolute_y}, \eqref{eq:p_norm} and \eqref{eq:q_norm}. Moreover, using \eqref{eq:norm_transmission}, we see that the transmission $t(u)=t_{\max}$ for all $u\in V(G)$. 

Equality in \eqref{eq:absolute_y} forces that $|\y|$ is a constant vector by \eqref{eq:quadratic_form}. Therefore 
\[\supp^+(\y)\sqcup \supp^-(\y)=V,\] 
and there exists a constant $c > 0$ such that 
\[y_u = 
\begin{cases}
c & \text{if }u\in \supp^+(\y);\\
-c & \text{if }u\in \supp^-(\y).
\end{cases}\]
Since $\1^\top \y = 0$, we see that 
\begin{equation}
  |\supp^+(\y)| = |\supp^-(\y)|.  
\end{equation}

Equality in \eqref{eq:p_norm} forces that $\dist(u,v)=1$ for all $u\in \supp^+(\y)$ and $v\in \supp^-(\y)$. Moreover, if $u,w\in \supp^+(\y)$ and $u\neq w$, then 
\[\dist(u,w) = \dist(u,v) + \dist(v,w) = 2\]
for any $v\in \supp^-(\y)$. Thus, $\supp^+(\y)$ is an independent set with all vertices of $\supp^+(\y)$ adjacent to all vertices of $\supp^-(\y)$.

Equality in \eqref{eq:q_norm} similarly forces that $\supp^-(\y)$ is an independent set. We conclude that $G \cong K_{m,m}$ where $m = |\supp^+(\y)|=|\supp^-(\y)|$. 
\end{proof}

Combining Claims \ref{claim:inequality_proof} and \ref{claim:equality_analysis} proves Theorem \ref{thm:spectral_gap_main}.

\section*{AI statement}

We acknowledge the use of AI tools during the ideation phase. We declare that the text is not AI-generated.

\bibliographystyle{plain}
\bibliography{reference}

\vspace{0.4cm}

\affl{Hitesh Kumar}{hitesh.kumar.math@gmail.com, hitesh\_kumar@sfu.ca}{Department of Mathematics, Simon Fraser University, Burnaby, Canada}

\affl{Kamal Lochan Patra}{klpatra@niser.ac.in}{School of Mathematical Sciences, National Institute of Science Education and Research (NISER), Bhubaneswar, An OCC of Homi Bhabha National Institute, Jatni, Khurda, Odisha 752050, India} 

\end{document}